\documentclass[12pt,a4paper]{article}%
\usepackage[utf8]{inputenc}
\usepackage{amsmath}
\usepackage{amsfonts}
\usepackage{amssymb}
\usepackage{graphicx}
\usepackage[space]{grffile}%
\providecommand{\U}[1]{\protect\rule{.1in}{.1in}}
\newtheorem{theorem}{Theorem}

\newtheorem{conjecture}[theorem]{Conjecture}
\newtheorem{corollary}[theorem]{Corollary}

\newtheorem{definition}[theorem]{Definition}
\newtheorem{example}[theorem]{Example}

\newtheorem{lemma}[theorem]{Lemma}

\newtheorem{proposition}[theorem]{Proposition}

\newenvironment{proof}[1][Proof]{\noindent\textbf{#1.} }{\ \hfill \rule{0.5em}{0.5em}\bigskip}
\graphicspath{{D:/Dropbox/Riste-Sedlar/MetricDim/Paper 6/Slike/}}

\begin{document}

\title{Metric dimensions vs. cyclomatic number of graphs with minimum degree at least two}
\author{Jelena Sedlar$^{1,3}$,\\Riste \v Skrekovski$^{2,3}$ \\[0.3cm] {\small $^{1}$ \textit{University of Split, Faculty of civil
engineering, architecture and geodesy, Croatia}}\\[0.1cm] {\small $^{2}$ \textit{University of Ljubljana, FMF, 1000 Ljubljana,
Slovenia }}\\[0.1cm] {\small $^{3}$ \textit{Faculty of Information Studies, 8000 Novo
Mesto, Slovenia }}\\[0.1cm] }
\maketitle

\begin{abstract}
The vertex (resp. edge) metric dimension of a connected graph $G,$ denoted by
$\mathrm{dim}(G)$ (resp. $\mathrm{edim}(G)$), is defined as the size of a
smallest set $S\subseteq V(G)$ which distinguishes all pairs of vertices
(resp. edges) in $G.$ Bounds $\mathrm{dim}(G)\leq L(G)+2c(G)$ and
$\mathrm{edim}(G)\leq L(G)+2c(G),$ where $c(G)$ is the cyclomatic number in
$G$ and $L(G)$ depends on the number of leaves in $G$, are known to hold for
cacti and are conjectured to hold for general graphs. In leafless graphs it
holds that $L(G)=0,$ so for such graphs the conjectured upper bound becomes
$2c(G).$ In this paper, we show that the bound $2c(G)$ cannot be attained by
leafless cacti, so the upper bound for such cacti decreases to $2c(G)-1$, and
we characterize all extremal leafless cacti for the decreased bound. We
conjecture that the decreased bound holds for all leafless graphs, i.e. graphs
with minimum degree at least two. We support this conjecture by showing that
it holds for all graphs with minimum degree at least three and that it is
sufficient to show that it holds for all $2$-connected graphs, and we also
verify the conjecture for graphs of small order.

\end{abstract}

\section{Introduction}

All graphs in this paper are tacitly assumed to be connected. We consider
several metric dimensions in connected graphs, and all of them involve the
notion of distance, so we define it here first. For a pair of vertices $u$ and
$v,$ the \emph{distance} $d(u,v)$ is defined as the length of the shortest
path connecting vertices $u$ and $v$. For a pair consisting of a vertex $u$
and an edge $e=vw,$ the \emph{distance} $d(u,e)$ is defined by $d(u,e)=\min
\{d(u,v),d(u,w)\}.$ Now, let $s$ be a vertex from $G$ and $X\subseteq V(G)\cup
E(G),$ we say that a pair $x$ and $x^{\prime}$ from $X$ is
\emph{distinguished} by $s$ if $d(s,x)\not =d(s,x^{\prime})$. We say that the
set $S\subseteq V(G)$ is a \emph{metric generator} of $X,$ if every pair
$x,x^{\prime}\in X$ is distinguished by at least one vertex from $S.$
Especially, if $S$ is a metric generator of $X=V(G)$ (resp. $X=E(G),$
$X=V(G)\cup E(G)$) then $S$ is a \emph{vertex} (resp. \emph{edge},
\emph{mixed}) \emph{metric generator}. The cardinality of a smallest vertex
(resp. edge, mixed) metric generator is the \emph{vertex} (resp. \emph{edge},
\emph{mixed}) \emph{metric dimension} of $G$ and it is denoted by
$\mathrm{dim}(G)$ (resp. $\mathrm{edim}(G),$ $\mathrm{mdim}(G)$).

The concept of vertex metric dimension is chronologically the first introduced
and it was studied related to the navigation systems \cite{HararyVertex} and
the problem of landmarks in networks \cite{KhullerVertex}. Since then this
variant of metric dimension was extensively investigated from various aspects
\cite{BuczkowskiVertex, ChartrandVertex, KleinVertex, MelterVertex}. Recently
it was noted that for some graphs the smallest vertex metric generators do not
distinguish all pairs of edges \cite{TratnikEdge}, so the notion of edge
metric dimension of a graph was introduced. This variant of metric dimension,
even though it is more recent, has also been quite studied \cite{GenesonEdge,
Knor, PeterinEdge, ZhuEdge, ZubrilinaEdge}. Finally, as a natural next step
the mixed metric dimension of a graph was introduced in \cite{KelencMixed} and
later further investigated in \cite{SedSkreTheta, SedSkrekMixed}. For this
paper particularly relevant is the line of investigation from papers
\cite{SedSkreBounds, SedSkreExtensionCactus, SedSkreUnicyclic} where an upper
bound on vertex and edge metric dimensions was established for unicyclic
graphs and further extended to cacti. There it was also conjectured that the
bound holds for connected graphs in general. In this paper we focus on cacti
without leaves and show that for such graphs the bound decreases by one and we
characterize which cacti without leaves attain this decreased bound.

For a vertex $v$ of a graph $G,$ the \emph{degree} $\deg(v)$ is the number of
vertices in $G$ adjacent to $v.$ The \emph{minimum degree} in a graph $G$ is
denoted by $\delta(G)$. If $\deg(v)=1,$ then we say $v$ is a \emph{leaf}. We
say that a set $S\subseteq V(G)$ is a \emph{vertex cut} if $G-S$ is
disconnected or trivial. If $S=\{v\}$ is a vertex cut, then we say $v$ is a
\emph{cut vertex}. The \emph{(vertex) connectivity} of a graph $G,$ denoted by
$\kappa(G),$ is defined as the cardinality of the smallest vertex cut in $G.$
A graph $G$ is $k$\emph{-connected} if $\kappa(G)\geq k.$ Notice that
$\kappa(G)\leq\delta(G)$, so $2$-connected graphs do not contain leaves. Also,
polycyclic cacti obviously have $\kappa(G)=1.$ A \emph{block} in a graph $G$
is any maximal $2$-connected subgraph of $G.$ A block $G_{i}$ in $G$ is
\emph{non-trivial,} if $G_{i}$ contains at least three vertices. We say that a
block $G_{i}$ of $G$ is an \emph{end-block} if $G_{i}$ contains precisely one
cut vertex from $G$.

Let $P=u_{1}\ldots u_{k}$ be an induced subpath of $G$ and let $v\in V(G)$ be
a vertex in $G$ of degree at least three. We say that $P$ is a \emph{thread}
in $G$ hanging at $v$ if the vertex $u_{1}$ is a leaf in $G$ and $u_{k}$ is
adjacent to $v$. We define the number $L(G)$ by
\[
L(G)=\sum_{v\in V(G),\ell(v)>1}(\ell(v)-1)
\]
where $\ell(v)$ is the number of threads hanging at a vertex $v$. Notice that
$\ell(v)>0$ may hold only for a vertex $v$ with degree $\geq3$. The
\emph{cyclomatic number} $c(G)$ of a graph $G$ is defined by $c(G)=\left\vert
E(G)\right\vert -\left\vert V(G)\right\vert +1.$ The following upper bounds on
$\mathrm{dim}(G)$ and $\mathrm{edim}(G)$\ are conjectured in
\cite{SedSkreExtensionCactus}, where it was also shown that the conjectured
bounds hold for graphs with edge disjoint cycles (also called \emph{cactus
graphs} or \emph{cacti}) and all extremal graphs characterized.

\begin{conjecture}
\label{Con_dim}Let $G$ be a connected graph. Then, $\mathrm{dim}(G)\leq
L(G)+2c(G).$
\end{conjecture}

\begin{conjecture}
\label{Con_edim}Let $G$ be a connected graph. Then, $\mathrm{edim}(G)\leq
L(G)+2c(G).$
\end{conjecture}

In this paper we further investigate these conjectures. First, we notice that
the attainment of the bounds in the class of cacti depends on the presence of
leaves in a graph and that the bounds cannot be attained by cacti without
leaves. The natural question that arises is what is the tight upper bound for
leafless cacti and does it extend to all graphs without leaves. We start the
investigation of that question by characterizing all cacti for which the first
smaller bound is attained, i.e. all cacti for which $\mathrm{dim}%
(G)=L(G)+2c(G)-1$ (resp. $\mathrm{edim}(G)=L(G)+2c(G)-1$).\ The direct
consequence is that this upper bound is attained by some leafless cacti, and
therefore it is a tight bound.

For all leafless graphs it holds that $L(G)=0,$ so the upper bound from
Conjectures \ref{Con_dim} and \ref{Con_edim} becomes $2c(G),$ and we suspect
that it cannot be attained by such graphs, just as it cannot be attained by
leafless cacti. Notice that a graph $G$ being leafless is equivalent to its
minimum degree being at least two. We state a formal conjecture that both
metric dimensions of such graphs are bounded above by $2c(G)-1.$ As a first
step towards the solution of this conjecture, we show that the upper bound
$2c(G)-1$ holds for both metric dimensions of all graphs with $\delta(G)\geq3$
and moreover the bound cannot be attained by them. This reduces the problem to
the class of graphs with $\delta(G)=2.$

Notice that graphs with $\delta(G)=2$ may have $\kappa(G)=1$ and
$\kappa(G)=2.$ For graphs $G$ with $\delta(G)=2$ and $\kappa(G)=1$ we further
show that if the upper bound $2c(G_{i})-1$ holds for a metric dimension of
every non-trivial block $G_{i}$ in $G$ distinct from a cycle, then $2c(G)-1$
is also an upper bound for the metric dimensions of $G.$ This further reduces
the problem to $2$-connected graphs, i.e. it only remains to prove that the
upper bound $2c(G)-1$ holds for metric dimensions of graphs with
$\kappa(G)=2.$ We leave this case open.

The similar results and conjectures for the mixed metric dimension are already
established in \cite{SedSkreTheta, SedSkrekMixed}.

\section{Preliminaries}

Let $G$ be a cactus graph, $C$ a cycle in $G$ and $v\in V(C).$ By $G_{v}(C)$
we denote the connected component of $G-E(C)$ which contains the vertex $v.$ A
vertex $v\in V(C)$ is said to be \emph{branch-active} if its degree is at
least $4$ or $G_{v}(C)$ contains a vertex of degree at least $3$ distinct from
$v$. Notice that in the case of $v$ being a branch-active vertex, there are
two vertices (resp. two edges) in $G_{v}(C)$ on the same distance from $v,$
which implies they will not be distinguished by a set $S\subseteq V(G)$ if
$G_{v}(C)$ does not contain a vertex from $S.$ We denote the number of
branch-active vertices on $C$ by $b(C).$

For a cactus graph $G$ with $c$ cycles $C_{1},\ldots,C_{c}$ we introduce the
notation
\[
B(G)=\sum_{i=1}^{c}\max\{0,2-b(C_{i})\}.
\]
We say that a cycle $C_{i}$ of a cactus graph is an \emph{end-cycle} if
$b(C_{i})=1.$ Notice that for every cycle $C_{i}$ of a cactus graph $G$ with
$c\geq2$ cycles, it holds that $b(C_{i})\geq1.$ Therefore, in such a graph
$B(G)\leq c$ with equality if and only if every cycle is an end-cycle.

\begin{figure}[h]
\begin{center}
\includegraphics[scale=0.75]{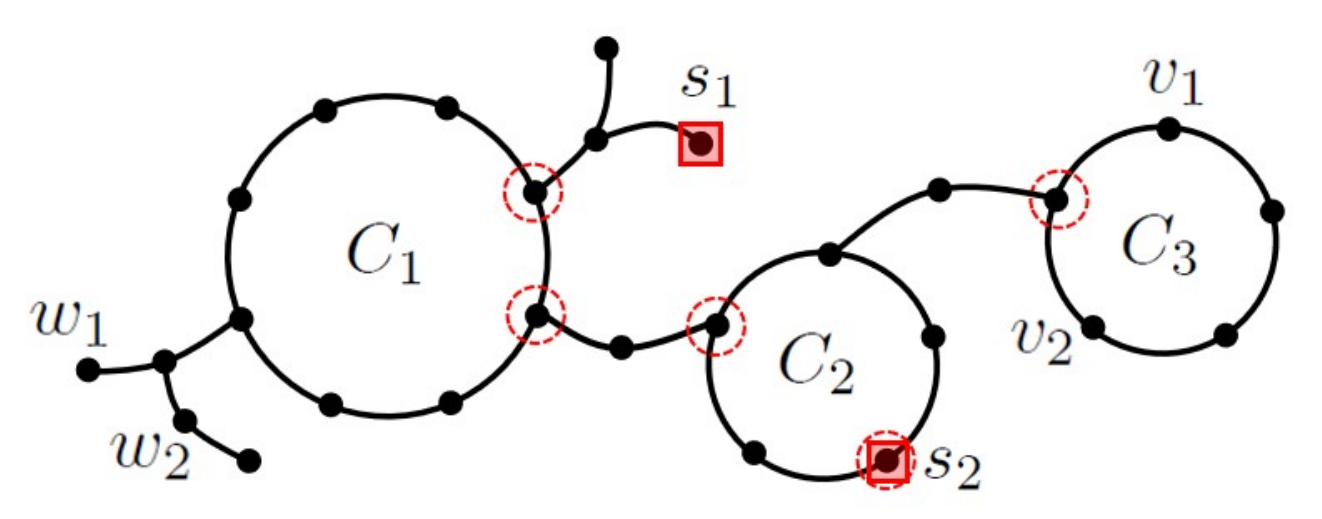}
\end{center}
\caption{A cactus graph $G$ with three cycles and a set $S=\{s_{1}%
,s_{2}\}\subseteq V(G)$ for which $S$-active vertices on each cycle are marked
by a dashed circle. The set $S$ is neither biactive, nor branch-resolving.
Since $S$ is not biactive on $C_{3}$, vertices $v_{1}$ and $v_{2}$ are not
distinguished by $S.$ And, since $S$ is not branch-resolving due to the two
$S$-free threads containing vertices $w_{1}$ and $w_{2}$ hanging at the same
vertex, vertices $w_{1}$ and $w_{2}$ are not distinguished by $S.$}%
\label{Fig_BBRset}%
\end{figure}

Let $G$ be a cactus graph, $S\subseteq V(G),$ $C$ a cycle in $G$ and $v\in
V(C).$ We say that $v$ is $S$\emph{-active} on $C$ if $G_{v}(C)$ contains a
vertex from $S.$ By $a_{S}(C)$ we denote the number of vertices on $C$ which
are $S$-active. A set $S\subseteq V(G)$ is \emph{biactive} in $G$ if every
cycle $C_{i}$ of $G$ contains at least two $S$-active vertices. An
$S$\emph{-free} thread is any thread in $G$ such that $S$ does not contain any
vertex of that thread. If for a set of vertices $S\subseteq V(G)$ it holds
that there are no two $S$-free threads in $G$ hanging at the same vertex, then
$S$ is a \emph{branch-resolving} set in $G$. Every biactive branch-resolving
set will be called shortly a \emph{BBR} set. It was established in
\cite{SedSkreExtensionCactus} that every vertex (resp. edge) metric generator
is a BBR set. The necessity of this condition is illustrated by Figure
\ref{Fig_BBRset}. Notice that for every smallest BBR set $S$ in a polycyclic
cactus graph $G$ it holds that $\left\vert S\right\vert =L(G)+B(G).$

\begin{figure}[h]
\begin{center}
\includegraphics[scale=0.75]{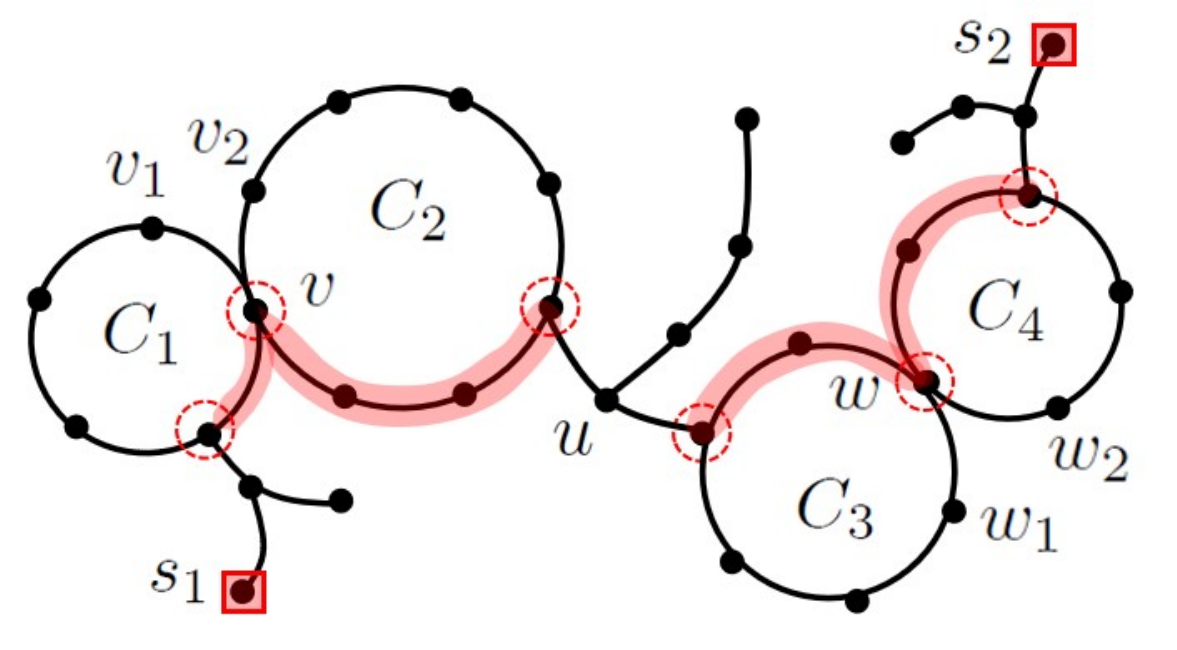}
\end{center}
\caption{A cactus graph $G$ with four cycles and a BBR set $S=\{s_{1}%
,s_{2}\}\subseteq V(G).$ On each of the four cycles the $S$-path is marked.}%
\label{Fig_Spath}%
\end{figure}

Throughout the paper for a given cactus graph $G$ and a BBR set $S\subseteq
V(G),$ we will assume that for a cycle $C$ with $V(C)=\{v_{0},\ldots
,v_{g-1}\},$ the vertices of $V(C)$ are denoted so that $v_{0}$ is $S$-active
and $k=\max\{i:v_{i}$ is $S$-active$\}$ is the smallest possible. Assuming
such labeling, the subpath of a cycle $C_{i}$ consisting of vertices
$v_{1}v_{2}\cdots v_{k-1}v_{k}$ is called an $S$\emph{-path} and denoted by
$P_{i}.$ The notion of $S$-path is illustrated by Figure \ref{Fig_Spath}.

\begin{figure}[ph]
\begin{center}
$%
\begin{array}
[c]{ll}%
\text{a) \raisebox{-1\height}{\includegraphics[scale=0.8]{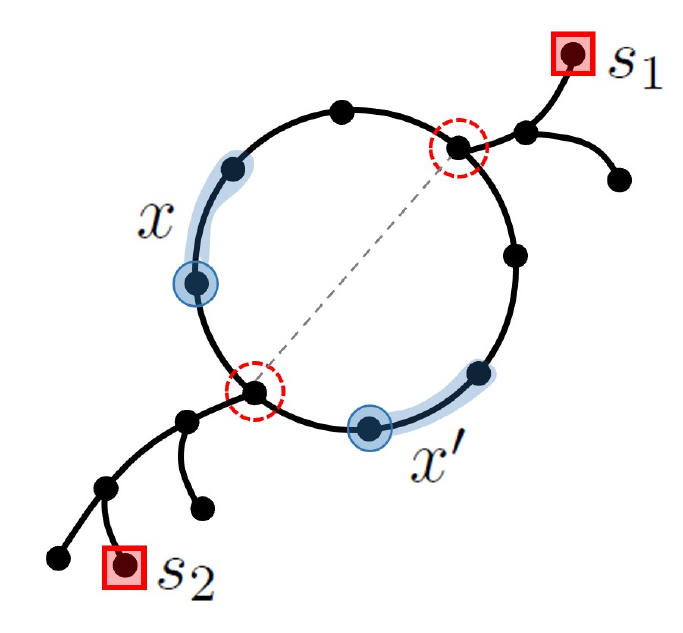}}} &
\text{b) \raisebox{-1\height}{\includegraphics[scale=0.8]{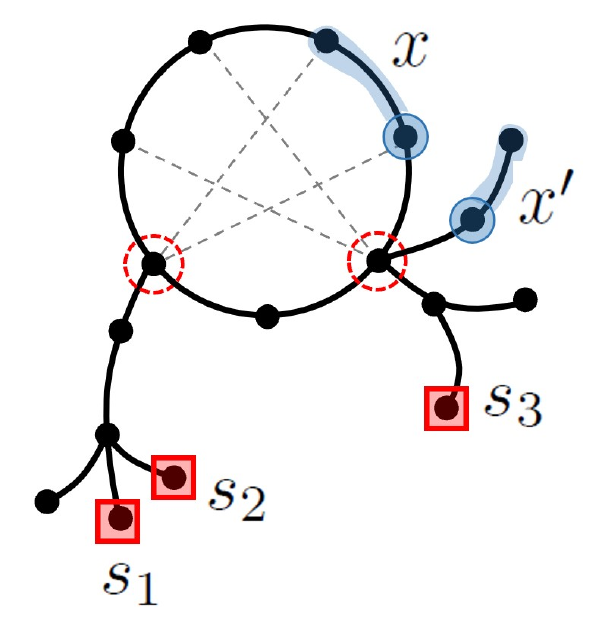}}}\\
\text{c) \raisebox{-1\height}{\includegraphics[scale=0.8]{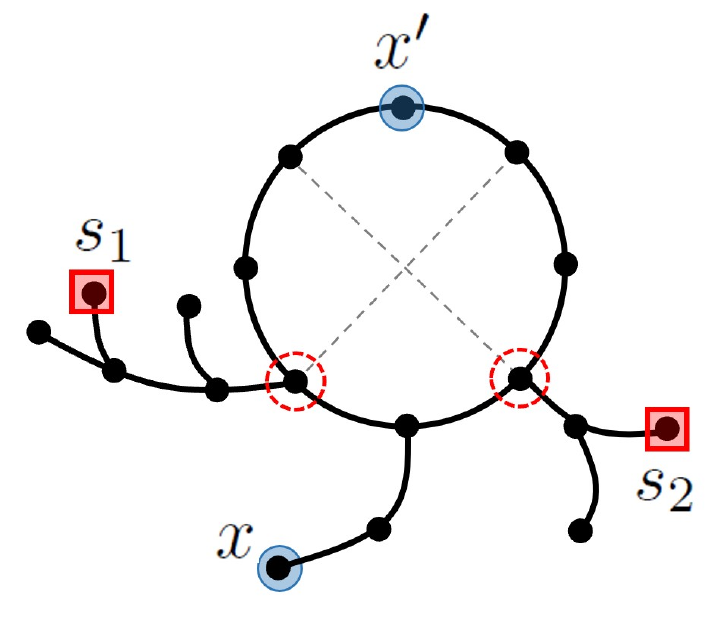}}} &
\text{d) \raisebox{-1\height}{\includegraphics[scale=0.8]{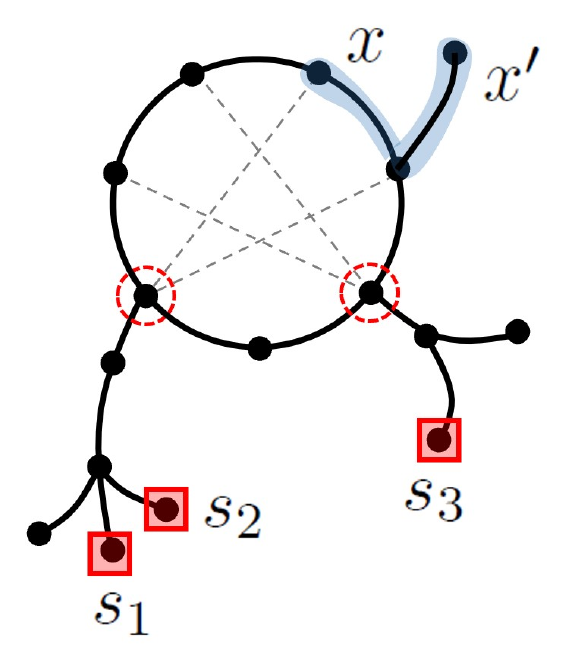}}}\\
\text{e) \raisebox{-1\height}{\includegraphics[scale=0.8]{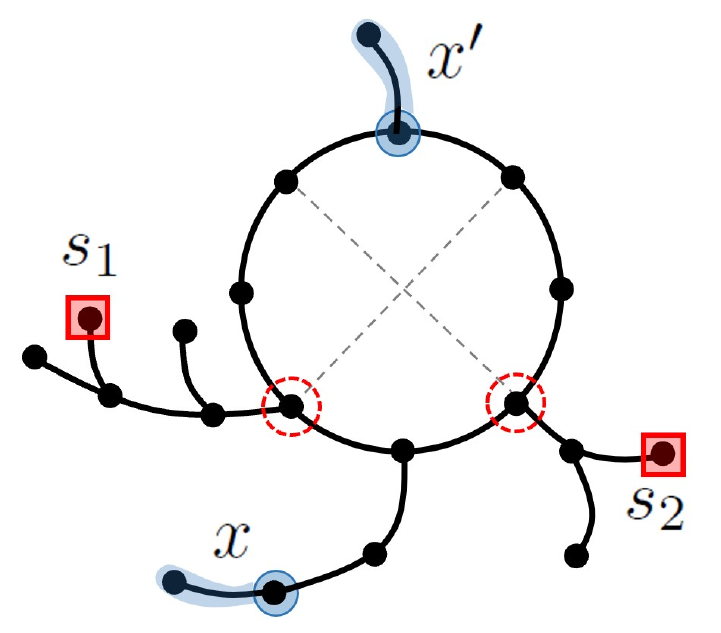}}} &
\text{f) \raisebox{-1\height}{\includegraphics[scale=0.8]{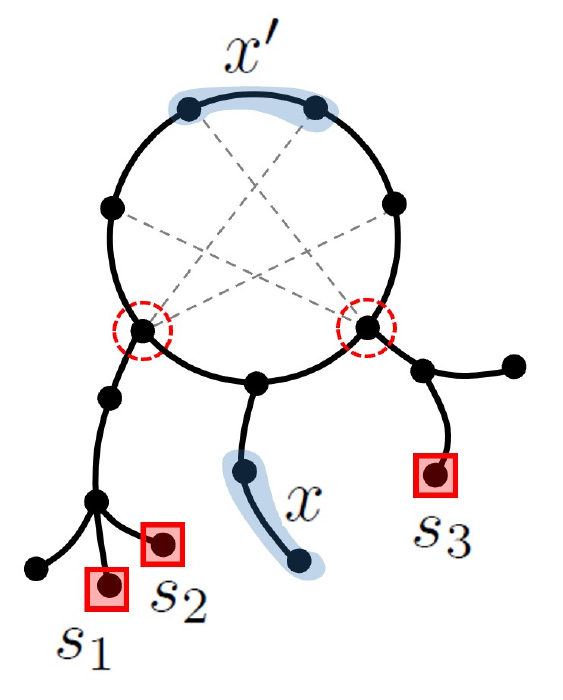}}}%
\end{array}
$
\end{center}
\caption{Each of the six graphs are unicyclic graphs with a BBR set of
vertices $S=\{s_{i}\}$ marked in them. The $S$-active vertices on the cycle
are marked by a dashed circle and dashed lines connect them to their
antipodals, as they delimit the area on cycle where the presence of an
$S$-free thread makes the configuration. The graph contains configurations: a)
$\mathcal{A}$, b) $\mathcal{B}$ and $\mathcal{D},$ c) $\mathcal{C}$, d)
$\mathcal{D},$ e) $\mathcal{E}$ on even cycle and $\mathcal{C}$, f)
$\mathcal{E}$ on odd cycle. By $x$ and $x^{\prime}$ a pair of vertices and/or
edges is denoted which is not distinguished by $S$. As illustrated by graph b)
configuration $\mathcal{B}$ is at the same time also $\mathcal{D},$ but graph
in d) shows that $\mathcal{D}$ does not have to be $\mathcal{B}.$ Similarly,
by the graph in e) it is illustrated that $\mathcal{E}$ is at the same time
$\mathcal{C},$ but the graph c) shows that the opposite does not hold, i.e.
$\mathcal{C}$ is not always $\mathcal{E}$.}%
\label{Fig_configurations}%
\end{figure}

Let us now introduce five configurations which a cycle in a cactus graph may
contain with respect to a BBR set $S.$

\begin{definition}
Let $G$ be a cactus graph, $C$ a cycle in $G$ of the length $g$, and $S$ a BBR
set in $G$. We say that the cycle $C$ \emph{with respect} to $S$
\emph{contains} configurations:

\begin{description}
\item {$\mathcal{A}$}. If $a_{S}(C)=2$, $g$ is even, and $k=g/2$;

\item {$\mathcal{B}$}. If $k\leq\left\lfloor g/2\right\rfloor -1$ and there is
an $S$-free thread hanging at a vertex $v_{i}$ for some $i\in\lbrack
k,\left\lfloor g/2\right\rfloor -1]\cup\lbrack\left\lceil g/2\right\rceil
+k+1,g-1]\cup\{0\}$;

\item {$\mathcal{C}$}. If $a_{S}(C)=2$, $g$ is even, $k\leq g/2$ and there is
an $S$-free thread of the length $\geq g/2-k$ hanging at a vertex $v_{i}$ for
some $i\in\lbrack0,k]$;

\item $\mathcal{D}$. If $k\leq\left\lceil g/2\right\rceil -1$ and there is an
$S$-free thread hanging at a vertex $v_{i}$ for some $i\in\lbrack
k,\left\lceil g/2\right\rceil -1]\cup\lbrack\left\lfloor g/2\right\rfloor
+k+1,g-1]\cup\{0\}$;

\item {$\mathcal{E}$}. If $a_{S}(C)=2$ and there is an $S$-free thread of the
length $\geq\left\lfloor g/2\right\rfloor -k+1$ hanging at a vertex $v_{i}$
with $i\in\lbrack0,k].$ Moreover, if $g$ is even, an $S$-free thread must be
hanging at the vertex $v_{j}$ with $j=g/2+k-i$.
\end{description}
\end{definition}

These configurations are illustrated by Figure \ref{Fig_configurations} and
originally introduced in \cite{SedSkreExtensionCactus}. The same figure also
illustrated why $S$ being a BBR set is only necessary, but not sufficient
condition for $S$ to be a metric generator. Consequently, when constructing a
smallest metric generator in a cactus graph $G$, one needs to start from a
smallest BBR set $S$ and then consider which of the cycles in $G$ contain one
of the configurations with respect to it, as for such cycles additional
vertices will have to be introduced into $S.$ Since it is obviously important
to consider if a cycle contains these configurations, for that purpose we
introduce the following definition.

\begin{definition}
We say that a cycle $C_{i}$ of a cactus graph $G$ is $\mathcal{ABC}%
$\emph{-negative} (resp. $\mathcal{ADE}$\emph{-negative}), if there exists a
smallest BBR set $S$ in $G$ such that $C_{i}$ does not contain any of the
configurations $\mathcal{A},$ $\mathcal{B},$ and $\mathcal{C}$ (resp.
$\mathcal{A},$ $\mathcal{D},$ and $\mathcal{E}$) with respect to $S.$
Otherwise, we say that $C_{i}$ is $\mathcal{ABC}$\emph{-positive} (resp.
$\mathcal{ADE}$\emph{-positive}). The number of $\mathcal{ABC}$-positive
(resp. $\mathcal{ADE}$-positive) cycles in $G$ is denoted by $c_{\mathcal{ABC}%
}(G)$ (resp. $c_{\mathcal{ADE}}(G)$).
\end{definition}

It is worth noting that there exists a smallest BBR set $S$ such that every
$\mathcal{ABC}$-negative (resp. $\mathcal{ADE}$-negative) does not contain the
three respective configurations with respect to $S$, as it was established in
\cite{SedSkreExtensionCactus}. There, it was also shown that the presence of
the three respective configurations on any of the cycles in $G$ is an obstacle
for $S$ to be a metric generator and that for each such cycle an additional
vertex has to be introduced to $S$ in order for it to become a metric
generator. But even that is only necessary and not sufficient condition for
$S$ to be a metric generator, as there may further occur a problem when cycles
share a vertex, for which we introduce the following definitions.

\begin{definition}
Let $G$ be a cactus graph with cycles $C_{1},\ldots,C_{c}$ and let $S$ be a
BBR set in $G.$ We say that a vertex $v\in V(C_{i})$ is \emph{vertex-critical}
(resp. \emph{edge-critical}) on $C_{i}$ with respect to $S$ if $v$ is an
end-vertex of the $S$-path $P_{i}$ and $\left\vert P_{i}\right\vert
\leq\left\lfloor g_{i}/2\right\rfloor -1$ (resp. $\left\vert P_{i}\right\vert
\leq\left\lceil g_{i}/2\right\rceil -1$).
\end{definition}

Notice that the notion of a vertex-critical and an edge-critical vertex
differs only on odd cycles.

\begin{definition}
Two distinct cycles $C_{i}$ and $C_{j}$ of a cactus graph $G$ are
\emph{vertex-critically incident} (resp. \emph{edge-critically incident}) with
respect to a BBR set $S\subseteq V(G)$ if $C_{i}$ and $C_{j}$ share a vertex
$v$ which is vertex-critical (resp. edge-critical) with respect to $S$ on both
$C_{i}$ and $C_{j}$.
\end{definition}

For illustration and motivation of these notions, consider the following example.

\begin{example}
\label{Example_criticalIncidence}Let $G$ be a graph and $S=\{s_{1},s_{2}\}$ a
set of vertices in $G$ as shown in Figure \ref{Fig_Spath}. None of the four
cycles in $G$ contains any of the five configurations. Vertex $v$ shared by
cycles $C_{1}$ and $C_{2}$ is vertex-critical on both cycles and it is also
edge-critical on both cycles. Therefore, cycles $C_{1}$ and $C_{2}$ are both
vertex- and edge-critically incident. Consequently, the pair of vertices
$v_{1}$ and $v_{2}$ and also the pair of edges $v_{1}v$ and $v_{2}v$ are not
distinguished by $S.$

On the other hand, the vertex $w$ shared by cycles $C_{3}$ and $C_{4}$ is
vertex-critical on both cycles, but edge-critical only on the cycle $C_{3}$
and not on $C_{4}.$ Therefore, cycles $C_{3}$ and $C_{4}$ are only
vertex-critically incident and not edge-critically incident. A consequence of
this is that the pair of vertices $w_{1}$ and $w_{2}$ on these two cycles is
not distinguished by $S,$ there is no a pair of indistingushed edges (notice
that $w_{1}w$ and $w_{2}w$ are distinguished by $S$).
\end{example}

A set $S\subseteq V(G)$ is a \emph{vertex cover} if it contains a least one
end-vertex of every edge in $G.$ The cardinality of a smallest vertex cover in
$G$ is the \emph{vertex cover number} of $G$ denoted by $\tau(G).$ Further,
let $G$ be a cactus graph and $S$ a smallest BBR set in $G,$ we say that $S$
is \emph{nice} if every $\mathcal{ABC}$-negative (resp. $\mathcal{ADE}%
$-negative) cycle $C_{i}$ in $G$ does not contain the three configurations
with respect to $S$ and the number of pairs of vertex-critically (resp.
edge-critically) incident cycles with respect to $S$ is the smallest possible.

Now, we define the \emph{vertex-incident graph} $G_{vi}$ (resp.
\emph{edge-incident graph} $G_{ei}$) as a graph containing a vertex for every
cycle in $G,$ where two vertices are adjacent if the corresponding cycles in
$G$ are $\mathcal{ABC}$-negative and vertex-critically incident (resp.
$\mathcal{ADE}$-negative and edge-critically incident) with respect to a nice
BBR set $S$. These notions are illustrated by the following example.

\begin{example}
Let $G$ be a graph and $S$ a set of vertices in $G$ as shown in Figure
\ref{Fig_Spath}. The vertex set of both $G_{vi}$ and $G_{ei}$ for the graph
$G$ is the same and consists of four vertices corresponding to the four cycles
in $G,$ i.e. $V(G_{vi})=V(G_{ei})=\{c_{1},c_{2},c_{3},c_{4}\}.$ Graphs
$G_{vi}$ and $G_{ei}$ differ in the set of edges, where $E(G_{vi}%
)=\{c_{1}c_{2},c_{3}c_{4}\}$ and $E(G_{ei})=\{c_{1}c_{2}\}.$
\end{example}

Now we can finally state the main results from \cite{SedSkreExtensionCactus}
which we need in this paper.

\begin{theorem}
\label{Tm_dim}Let $G$ be a cactus graph. Then
\[
\mathrm{dim}(G)=L(G)+B(G)+c_{\mathcal{ABC}}(G)+\tau(G_{vi}),
\]
and
\[
\mathrm{edim}(G)=L(G)+B(G)+c_{\mathcal{ADE}}(G)+\tau(G_{ei}).
\]

\end{theorem}

The direct consequence of the exact formulas for metric dimensions of cacti is
the following simple upper bound for both metric dimensions, also from
\cite{SedSkreExtensionCactus}.

\begin{corollary}
\label{Cor_bound}Let $G$ be a cactus graph with $c\geq2$ cycles. Then
\[
\mathrm{dim}(G)\leq L(G)+2c\text{ (resp. }\mathrm{edim}(G)\leq L(G)+2c\text{)}%
\]
with equality holding if and only if every cycle in $G$ is an $\mathcal{ABC}%
$-positive (resp. $\mathcal{ADE}$-positive) end-cycle.
\end{corollary}

\section{From leafless cacti to leafless general graphs}

Considering further graphs for which the metric dimensions are bounded above
by $L(G)+2c(G),$ notice that according to Corollary \ref{Cor_bound}, an
attainment of the bound in cacti depends on every cycle being $\mathcal{ABC}%
$-positive (resp. $\mathcal{ADE}$-positive) and an end-cycle. The presence of
the configurations $\mathcal{B},$ $\mathcal{C},$ $\mathcal{D}$ and
$\mathcal{E}$ in cacti by definition implies the existence of threads hanging
at a cycle, and therefore leaves. We conclude that cycles in a leafless cactus
graph cannot contain these configurations. As for configuration $\mathcal{A}$,
a cycle in a leafless cactus graph may contain this configuration, but only if
it is not an end-cycle, which means that the bound $L(G)+2c(G)$ again cannot
be attained by such a cactus graph.

The fact that leafless cacti do not attain the bound might be interesting when
considering general connected graphs without leaves, i.e. all graphs in which
$\delta(G)\geq2.$ This motivates us to investigate which cacti are nearly
extremal, i.e. for which $\mathrm{dim}(G)=L(G)+2c-1$ (resp. $\mathrm{edim}%
(G)=L(G)+2c-1$).

\begin{proposition}
\label{Prop_nearlyExtremal}Let $G$ be a cactus graph with $c\geq2$ cycles.
Then $\mathrm{dim}(G)=L(G)+2c-1$ if and only if one of the following holds:

\begin{enumerate}
\item every cycle in $G$ is an end-cycle, at most $c-1$ cycles are
$\mathcal{ABC}$-positive and all remaining cycles are pairwise
vertex-critically incident;

\item precisely $c-1$ cycles in $G$ are end-cycles and every cycle in $G$ is
$\mathcal{ABC}$-positive.
\end{enumerate}
\end{proposition}

\begin{proof}
Notice that in a cactus graph $G$ with at least two cycles, it holds that
$b(C_{i})\geq1$ for every cycle $C_{i}$ in $G,$ and therefore $B(G)\leq c.$
Also, by definition a vertex of $G_{vi}$ is incident to an edge in $G_{vi}$
only if it corresponds to an $\mathcal{ABC}$-negative cycle in $G,$ which
implies $c_{\mathcal{ABC}}(G)+\tau(G_{vi})\leq c.$ Therefore, for a cactus
graph $G$ the equality $\mathrm{dim}(G)=L(G)+2c-1$ will hold if and only if
either $B(G)=c$ and $c_{\mathcal{ABC}}(G)+\tau(G_{vi})=c-1$ or $B(G)=c-1$ and
$c_{\mathcal{ABC}}(G)+\tau(G_{vi})=c.$

Since the vertex cover number of any graph $G$ with $n$ vertices is at most
$n-1,$ this implies $\tau(G_{vi})<c.$ Also, vertex cover number in a graph $G$
with $n$ vertices will be equal to $n-1$ if and only if $G=K_{n}.$ This
further implies that if $G_{vi}$ is a graph on $q$ vertices (which correspond
to $q$ cycles in $G$), then $\tau(G_{vi})=q-1$ if and only if $G_{vi}$ is a
complete graph which is further equivalent to all $q$ corresponding cycles in
$G$ being pairwise vertex-critically incident. One useful consequence of this
observation is that $\tau(G_{vi})>0$ implies $c_{\mathcal{ABC}}(G)+\tau
(G_{vi})<c.$ Finally, for $\tau(G_{vi})$ to be strictly positive, there must
exist in $G$ at least two vertex-critically incident cycles.

Now we can consider separately the two conditions under which the equality
$\mathrm{dim}(G)=L(G)+2c-1$ will hold:

\begin{itemize}
\item $B(G)=c$\emph{ and }$c_{\mathcal{ABC}}(G)+\tau(G_{vi})=c-1.$ This
happens if and only if every cycle in $G$ is an end-cycle and either
$c_{\mathcal{ABC}}(G)=c-1$ or $c_{\mathcal{ABC}}(G)\leq c-2$ and all remaining
cycles in $G$ which are not $\mathcal{ABC}$-positive are pairwise
vertex-critically incident;

\item $B(G)=c-1$\emph{ and }$c_{\mathcal{ABC}}(G)+\tau(G_{vi})=c.$ This
happens if and only if precisely $c-1$ cycles in $G$ are end cycles and every
cycle in $G$ is $\mathcal{ABC}$-positive.
\end{itemize}
\end{proof}

Notice that in Proposition \ref{Prop_nearlyExtremal}.1, when a cactus graph
$G$ contains precisely $c-1$ cycles which are $\mathcal{ABC}$-positive, the
requirement for the vertex-critical incidence does not really apply. So there
are really no additional requirements on the remaining cycle. For the
illustration of the above result, let us consider the following example.

\begin{example}
Let $G_{1}$, $G_{2}$ and $G_{3}$ be cacti from a), b) and c) of Figure
\ref{Fig_threeGraphs}, respectively. For each $G_{i}$ it holds that
$c(G_{i})=3$ and $L(G_{i})=1.$

\begin{enumerate}
\item In $G_{1}$ all three cycles are end-cycles, cycles $C_{1}$ and $C_{2}$
contain configuration $\mathcal{B}$ due to a thread hanging at their only
branch-active vertex and therefore they are $\mathcal{ABC}$-positive. Cycle
$C_{3}$ does not contain any of the configurations and therefore it is
$\mathcal{ABC}$-negative. According to Proposition \ref{Prop_nearlyExtremal}.1
we conclude $\mathrm{dim}(G_{1})=L(G_{1})+2c-1=6.$

\item In $G_{2}$ all three cycles are end-cycles, cycle $C_{1}$ contains
configuration $\mathcal{B}$ due to a thread hanging at its only branch-active
vertex and therefore it is $\mathcal{ABC}$-positive, cycles $C_{2}$ and
$C_{3}$ do not contain any of the configurations and therefore they are
$\mathcal{ABC}$-negative, but they are vertex-critically incident. Similarly,
according to Proposition \ref{Prop_nearlyExtremal}.2 we conclude
$\mathrm{dim}(G_{2})=L(G_{2})+2c-1=6.$

\item In $G_{3}$ cycles $C_{1}$ and $C_{3}$ are end-cycles, but $C_{2}$ is
not. All three cycles contain configuration $\mathcal{B}$ and therefore
$\mathcal{ABC}$-positive. In a similar way, according to Proposition
\ref{Prop_nearlyExtremal}.3, we conclude $\mathrm{dim}(G_{3})=L(G_{3}%
)+2c-1=6.$
\end{enumerate}
\end{example}

\begin{figure}[ph]
\begin{center}
$%
\begin{array}
[c]{l}%
\text{a) \raisebox{-1\height}{\includegraphics[scale=0.8]{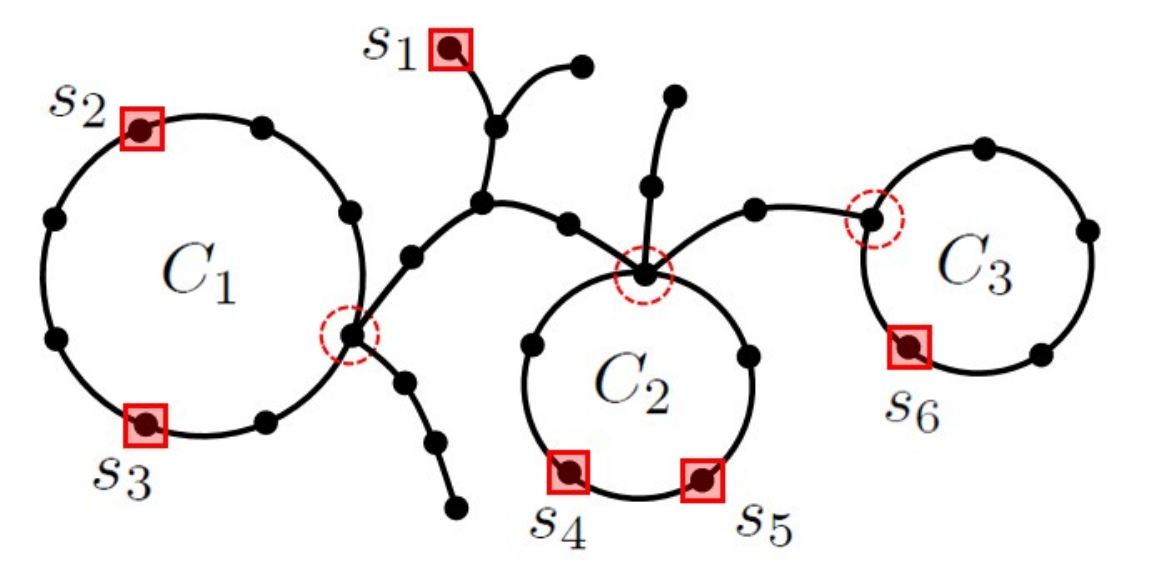}}}\\
\text{b) \raisebox{-1\height}{\includegraphics[scale=0.8]{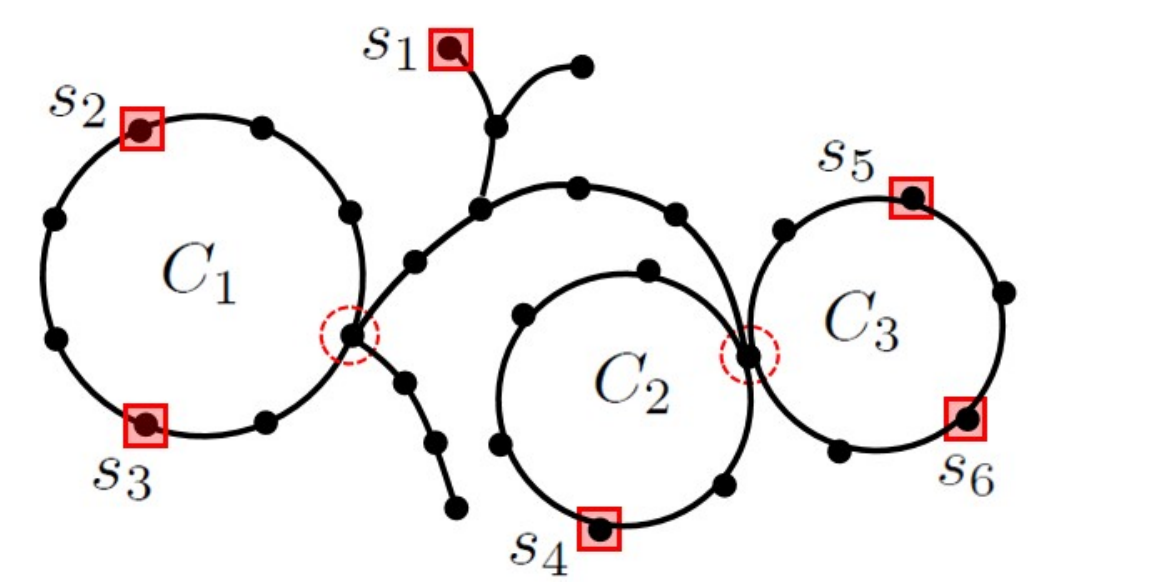}}}\\
\text{c) \raisebox{-1\height}{\includegraphics[scale=0.8]{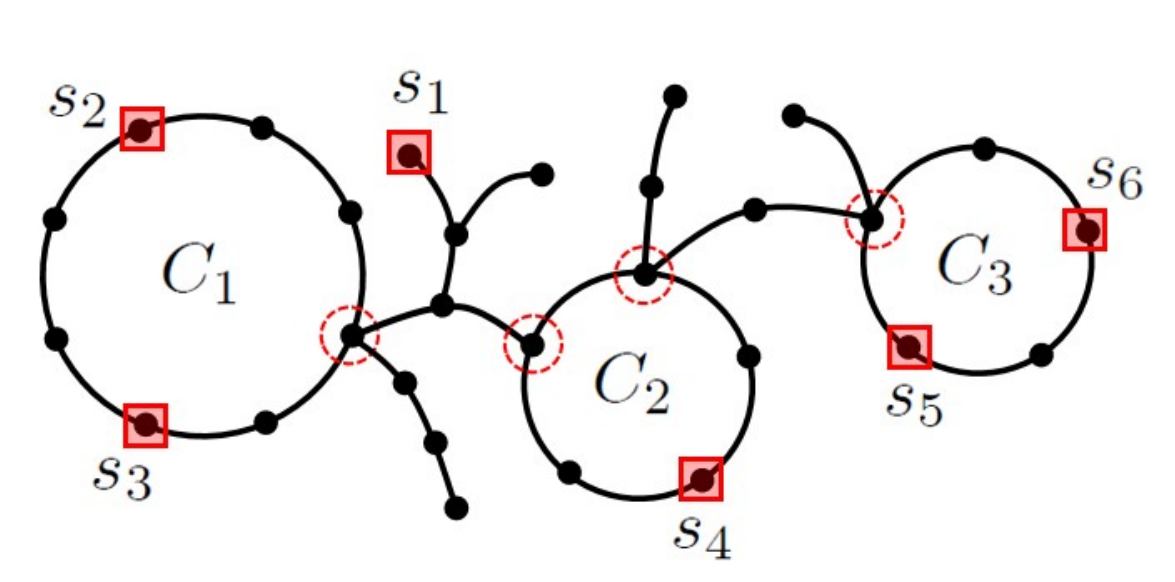}}}%
\end{array}
$
\end{center}
\caption{Three distinct cacti $G_{1},$ $G_{2}$ and $G_{3}$, each with $c=3$
cycles, $L(G_{i})=1$ and $\mathrm{dim}(G_{i})=L(G_{i})+2c-1=6.$ A smallest
vertex metric generator $S=\{s_{1},s_{2},\ldots\}$ is marked in each of the
graphs.}%
\label{Fig_threeGraphs}%
\end{figure}

In the light of Proposition \ref{Prop_nearlyExtremal}, we can now consider
leafless cacti, which might be an indication what happens for all graphs with
$\delta(G)\geq2$. We first need to introduce a special class of leafless
cacti. If a graph $G$ is comprised of cycles which all share one vertex, then
we say $G$ is a \emph{daisy} graph. A cycle of a daisy graph is also called a
\emph{petal}. The \emph{center} of a daisy graph $G$ is the only vertex from
$G$ of degree $>2$. Notice that a daisy graph by definition is a cactus graph
without leaves. An example of a daisy graph is shown in Figure \ref{Fig_daisy}.

\begin{figure}[h]
\begin{center}
$%
\begin{array}
[c]{ll}%
\text{a) \raisebox{-1\height}{\includegraphics[scale=0.8]{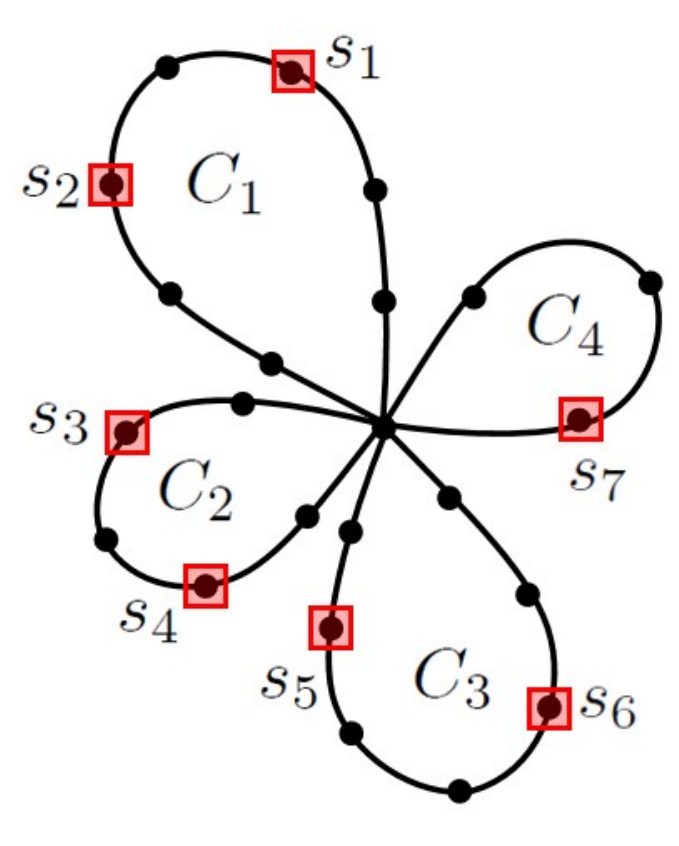}}} &
\text{b) \raisebox{-1\height}{\includegraphics[scale=0.8]{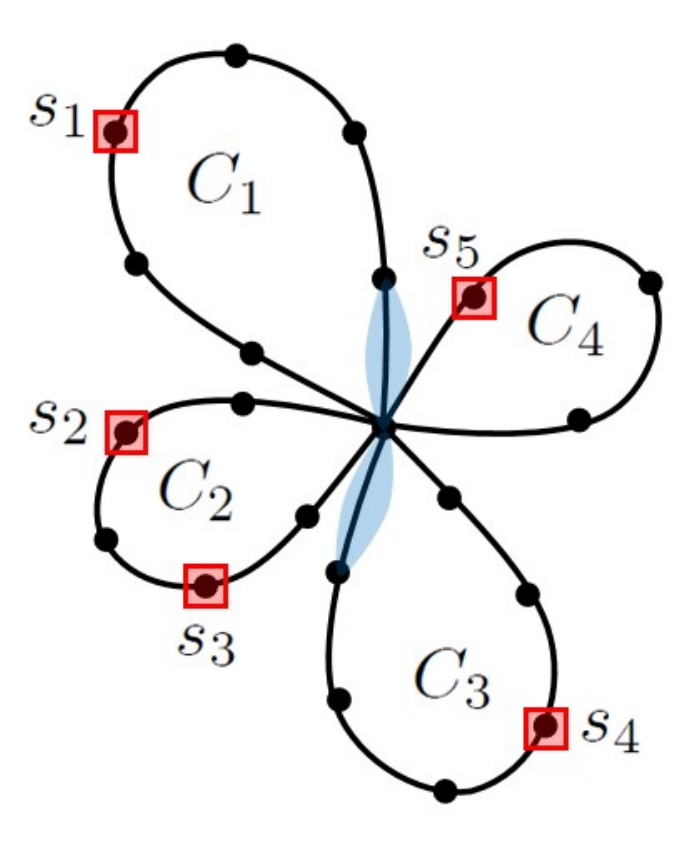}}}%
\end{array}
$
\end{center}
\caption{Two distinct daisy graphs $G$, each with four petals and a set of
vertices $S=\{s_{1},s_{2},\ldots\}$ marked in them. In a daisy graph: a) all
petals are even, b) petals $C_{1}$ and $C_{3}$ are odd, $C_{2}$ and $C_{4}$
are even. The set $S$ is: a) both a vertex and an edge metric generator, b)
only a vertex metric generator (the undistinguished pair of edges is marked in
a graph).}%
\label{Fig_daisy}%
\end{figure}

\begin{proposition}
\label{Prop_leflessDim}Let $G$ be a cactus graph with $c\geq2$ cycles and
without leaves. Then $\mathrm{dim}(G)\leq2c-1$ with equality if and only if
$G$ is a daisy graph without odd petals.
\end{proposition}

\begin{proof}
First, notice that for a leafless graph $G$ it holds that $L(G)=0.$ Therefore,
the bound $L(G)+2c$ for leafless cacti becomes $2c.$ According to Corollary
\ref{Cor_bound}, this bound is attained if and only if every cycle in $G$ is
an $\mathcal{ABC}$-positive end-cycle. The definition of the five
configurations implies the existence of a thread, and therefore a leaf, in $G$
for all configurations except $\mathcal{A}.$ Consequently, any cycle in a
leafless cactus graph can contain only configuration $\mathcal{A},$ but then
$C$ must not be an end-cycle, so according to Corollary \ref{Cor_bound} the
bound $2c$ cannot be attained in a class of leafless cacti, which implies
$\mathrm{dim}(G)\leq2c-1.$

Next, we investigate if this new bound $2c-1$ is attained by some leafless
cactus graph. Recall again that a cycle in a leafless cactus graph $G$ can
contain only configuration $\mathcal{A}$ and that only on a cycle which is not
an end-cycle. Proposition \ref{Prop_nearlyExtremal} implies that for a
leafless cactus graph $G$ it holds that $\mathrm{dim}(G)=2c-1$ if and only if
every cycle in $G$ is an end-cycle, at most $c-2$ cycles are $\mathcal{ABC}%
$-positive and all remaining cycles are pairwise vertex-critically incident.
To be more precise, since an end-cycle in $G$ cannot contain any of the
configurations, this characterization needs to be interpreted as
$\mathrm{dim}(G)=2c-1$ if and only if every cycle in $G$ is an end-cycle and
all of them are pairwise vertex-critically incident. Since vertex-critically
incident pair of cycles share a vertex, this implies that $G$ must be a daisy
graph. That is necessary, but it is not sufficient as we will show that odd
end-cycles cannot be vertex-critically incident with any other cycle in $G.$

Namely, recall that cycle $C_{i}$ is vertex-critically incident to another
cycle if $C_{i}$ shares a vertex $v$ with that other cycle, such that $v$ is
an end-vertex of path $P_{i}$ and the length of path $P_{i}$ is $\left\vert
P_{i}\right\vert \leq\left\lfloor g_{i}/2\right\rfloor -1.$ Yet, on the odd
end-cycle $C_{i}$ we can always choose a nice smallest BBR set $S,$ so that
$S$ contains an antipodal of the only branch-active vertex on $C_{i}.$ In that
case the length of the path $P_{i}$ on the cycle $C_{i}$ will be $\left\vert
P_{i}\right\vert =\left\lfloor g_{i}/2\right\rfloor >\left\lfloor
g_{i}/2\right\rfloor -1,$ so an odd end-cycle $C_{i}$ cannot be
vertex-critically incident to any other cycle in $G$, which concludes the proof.
\end{proof}

The result of the previous proposition is illustrated by Figure
\ref{Fig_daisy}. The statements and the proofs for the edge metric dimensions
are analogous.

\begin{proposition}
\label{Prop_nearlyExtremalEdge}Let $G$ be a cactus graph with $c\geq2$ cycles.
Then $\mathrm{edim}(G)=L(G)+2c-1$ if and only if one of the following holds:

\begin{enumerate}
\item every cycle in $G$ is an end-cycle, at most $c-1$ cycles are
$\mathcal{ADE}$-positive and all remaining cycles are pairwise
vertex-critically incident;

\item precisely $c-1$ cycles in $G$ are end-cycles and every cycle in $G$ is
$\mathcal{ADE}$-positive.
\end{enumerate}
\end{proposition}

\begin{proof}
The proof is analogous to the proof of Proposition \ref{Prop_nearlyExtremal}.
\end{proof}

Similarly as with the vertex metric dimension, we can now consider the edge
metric dimension of leafless cacti.

\begin{proposition}
Let $G$ be a cactus graph with $c\geq2$ cycles and without leaves. Then
$\mathrm{edim}(G)\leq2c-1$ with equality if and only if $G$ is a daisy graph.
\end{proposition}

\begin{proof}
The proof is analogous to the proof of Proposition \ref{Prop_leflessDim}. The
only difference is that a cycle $C_{i}$ in a cactus graph $G$ is an
edge-critically incident to another cycle if $C_{i}$ shares a vertex $v$ with
that other cycle such that $v$ is an end-vertex of the path $P_{i}$ of the
length $\left\vert P_{i}\right\vert \leq\left\lceil g_{i}/2\right\rceil -1.$
The difference in bound on $\left\vert P_{i}\right\vert $ which now contains
the ceiling of $g_{i}/2$ instead of the floor of $g_{i}/2$ which was the case
with the vertex dimension, implies that now we cannot choose a smallest BBR
set $S$ such that the length of $P_{i}$ is certainly longer than required.
Consequently, now any end-cycle can be edge-critically incident to another
cycle independently of its parity.
\end{proof}

The difference in extremal daisy graphs for the vertex and the edge metric
dimension, where for the vertex dimension only daisy graphs with even petals
are extremal and for the edge metric dimension all daisy graphs are extremal,
is illustrated by Figure \ref{Fig_daisy}. The upper bound for metric
dimensions of leafless cacti leads us to the opinion that for general leafless
graphs, the following may hold.

\begin{conjecture}
\label{Con_dim_leaves}Let $G\not =C_{n}$ be a graph with minimum degree
$\delta(G)\geq2$. Then, $\mathrm{dim}(G)\leq2c(G)-1.$
\end{conjecture}

\begin{conjecture}
\label{Con_edim_leaves}Let $G\not =C_{n}$ be a graph with minimum degree
$\delta(G)\geq2$. Then, $\mathrm{edim}(G)\leq2c(G)-1.$
\end{conjecture}

Both conjectures were tested both systematically and stochastically for graphs
of smaller order.

\section{Reduction to $2$-connected graphs}

As a first step towards solving Conjectures \ref{Con_dim_leaves} and
\ref{Con_edim_leaves}, in the following proposition we will show that they
hold for graphs with $\delta(G)\geq3$.

\begin{proposition}
\label{Prop_connected}Let $G$ be a graph with minimum degree $\delta(G)\geq3.$
Then $\mathrm{dim}(G)<2c(G)-1$ and $\mathrm{edim}(G)<2c(G)-1.$
\end{proposition}

\begin{proof}
From $\delta(G)\geq3$ we obtain $2m=\sum_{v\in V(G)}\deg(v)\geq n\delta
(G)\geq3n,$ which is equivalent to $n-1\leq2m-2n-1.$ Obviously, a set
$S\subseteq V(G)$ with $\left\vert S\right\vert =n-1$ is both a vertex and an
edge metric generator in $G,$ so%
\[
\mathrm{dim}(G)\leq n-1<2m-2n+1=2c(G)-1.
\]
A similar argument holds for $\mathrm{edim}(G)$.
\end{proof}

The above proposition implies that it only remains to show that Conjectures
\ref{Con_dim_leaves} and \ref{Con_edim_leaves} hold for graphs with
$\delta(G)=2.$ Notice that graphs $G$ with $\delta(G)=2$ may have
$\kappa(G)=1$ or $\kappa(G)=2.$ If $\kappa(G)=1,$ then $\delta(G)=2$ implies
that $G$ contains at least two non-trivial blocks. The natural question that
arises is if the problem can further be reduced to blocks in such a graph. We
will show that it can, i.e. if Conjecture \ref{Con_dim_leaves} (resp.
Conjecture \ref{Con_edim_leaves}) holds for every non-trivial block $G_{i}$ of
$G$ distinct from cycle, then it also holds for $G.$ In order to show this, we
first need the following two lemmas.

\begin{lemma}
\label{Lemma_c}Let $G$ be any graph. Then
\[
c(G)=c(G_{1})+\cdots+c(G_{s}),
\]
where $G_{1},\ldots,G_{s}$ are the blocks in $G.$
\end{lemma}

\begin{proof}
Notice that $c(K_{2})=0,$ so let $G_{1},\ldots,G_{q}$ be all non-trivial
blocks in $G$. Also, for any spanning subtree $T$ of a graph $G$, it holds
that
\[
c(G)=\left\vert E(G)\right\vert -\left\vert E(T)\right\vert .
\]
Now, let $T$ be a spanning tree in $G$. Let us denote $E^{c}=E(G)\backslash
E(T)$ and $E_{i}^{c}=E^{c}\cap E(G_{i}).$ Let us denote $T_{i}=G_{i}-E_{i}%
^{c}.$ Obviously, $T_{i}$ is a subgraph of $T$ and therefore a tree. Since
$V(T_{i})=V(G_{i}),$ it follows that $T_{i}$ is spanning tree of $G_{i}$, so
$c(G_{i})=\left\vert E_{i}^{c}\right\vert .$ Thus we have
\[
c(G_{1})+\cdots+c(G_{q})=\left\vert E_{1}^{c}\right\vert +\cdots+\left\vert
E_{q}^{c}\right\vert =\left\vert E^{c}\right\vert =\left\vert E(G)\right\vert
-\left\vert E(T)\right\vert =c(G)
\]
and the claim is proven.
\end{proof}

In the above lemma we considered how the cyclomatic number of a graph $G$
relates to cyclomatic number of its non-trivial block. In the next lemma we
will consider the same for metric dimensions.

\begin{lemma}
\label{Lemma_dim}Let $G\not =C_{n}$ be a graph with $\delta(G)\geq2$. Let
$G_{1},\ldots,G_{q}$ be all non-trivial blocks in $G$ and $p$ of them distinct
from cycle. Then%
\[
\mathrm{dim}(G)\leq\mathrm{dim}(G_{1})+\cdots+\mathrm{dim}(G_{q})+p-1
\]
and
\[
\mathrm{edim}(G)\leq\mathrm{edim}(G_{1})+\cdots+\mathrm{edim}(G_{q})+p-1.
\]

\end{lemma}

\begin{proof}
Without loss of generality we may assume that blocks of $G$ are denoted so
that $G_{i}$ is a cycle if and only if $i>p.$ For $i=1,\ldots,p,$ let $S_{i}$
be a vertex (resp. an edge) metric generator in $G_{i}.$ For $i=p+1,\ldots,q,$
a block $G_{i}$ is a cycle, so $\mathrm{dim}(G_{i})=2$ (resp. $\mathrm{edim}%
(G_{i})=2$). In this case when non-trivial block $G_{i}$ is a cycle, we will
not choose for $S_{i}$ a vertex (resp. an edge) metric generator in $G_{i},$
but a smaller set consisting of precisely one vertex. To be more precise, for
$i=p+1,\ldots,q$ we define $S_{i}=\{v_{i}\}$ so that $v_{i}$ is any vertex
from $G_{i}$ if $G_{i}$ is not an end-block, otherwise we choose for $v_{i}$ a
non-cut vertex from $G_{i}$ which in the case when $G_{i}$ is an even cycle is
not an antipodal vertex of the only cut vertex in $G_{i}$. Let $S=S_{1}%
\cup\cdots\cup S_{q}$ and notice that
\[
\left\vert S\right\vert =\mathrm{dim}(G_{1})+\cdots+\mathrm{dim}(G_{q})-(q-p)
\]
(resp. $\left\vert S\right\vert =\mathrm{edim}(G_{1})+\cdots+\mathrm{edim}%
(G_{q})-(q-p)$). Also, notice that every end-block $G_{i}$ contains a vertex
$s_{i}\in S$ which is not a cut vertex in $G$.

Let $x$ and $x^{\prime}$ be a pair of vertices (resp. edges) in $G$. We
proceed with the following claims.

\bigskip\noindent\textbf{Claim A.} \emph{If }$x$\emph{ and }$x^{\prime}$\emph{
do not belong to two distinct non-trivial blocks of }$G,$\emph{ then }%
$x$\emph{ and }$x^{\prime}$\emph{ are distinguished by }$S.$

\smallskip\noindent If $x$ and $x^{\prime}$ belong to a same non-trivial block
$G_{i},$ then $x$ and $x^{\prime}$ are distinguished by $S_{i}$ in $G_{i}$ and
therefore also by $S$ in $G$ since every block is an isometric subgraph of
$G.$ Otherwise, at least one of $x$ and $x^{\prime},$ say $x,$ is a cut vertex
(resp. cut edge) in $G.$ This means $G-x$ has at least two connected
components, at least one of which does not contain $x^{\prime}.$ Since
$\delta(G)\geq2,$ each of the components of $G-x$ contains a non-trivial
end-block. Let $G_{j}$ be a non-trivial end-block of the connected component
of $G-x$ which does not contain $x^{\prime}$ and $s\in V(G_{j})\cap S.$ Notice
that the shortest path from $s$ to $x^{\prime}$ leads through $x,$ so $x$ and
$x^{\prime}$ are distinguished by $S$ which proves the claim.

\bigskip\noindent\textbf{Claim B.} \emph{If }$x$\emph{ and }$x^{\prime}$\emph{
belong to two distinct non-trivial blocks }$G_{i}$\emph{ and }$G_{j}$\emph{ of
}$G,$\emph{ such that }$V(G_{i})\cap V(G_{j})=\phi,$\emph{ then }$x$\emph{ and
}$x^{\prime}$\emph{ are distinguished by }$S.$

\smallskip\noindent Let $v\in V(G_{i})$ and $v^{\prime}\in V(G_{j})$ be two
cut vertices in $G$ such that the shortest path from $x$ to $x^{\prime}$ leads
through $v$ and $v^{\prime}.$ Since $V(G_{i})\cap V(G_{j})=\phi,$ it follows
that $v\not =v^{\prime}.$ Since every end-block in $G$ contains a vertex from
$S,$ it follows that there must exist an end-block $G_{k}$ and a vertex
$s^{\prime}\in V(G_{k})\cap S,$ such that the shortest path from $v$ to
$s^{\prime}$ leads through $v^{\prime}.$ Similarly, there must exist a vertex
$s\in S,$ such that the shortest path from $v^{\prime}$ to $s$ leads through
$v.$ Assume that $x$ and $x^{\prime}$ are not distinguished by $s^{\prime}$ in
$G,$ i.e. $d(x,s^{\prime})=d(x^{\prime},s^{\prime})$. Then from%
\begin{align*}
d(x,s^{\prime})  &  =d(x,v^{\prime})+d(v^{\prime},s^{\prime}%
)=d(x,v)+d(v,v^{\prime})+d(v^{\prime},s^{\prime})\\
d(x^{\prime},s^{\prime})  &  \leq d(x^{\prime},v^{\prime})+d(v^{\prime
},s^{\prime}),
\end{align*}
we obtain $d(x,v)+d(v,v^{\prime})\leq d(x^{\prime},v^{\prime}).$ The fact
$v\not =v^{\prime}$ implies $d(v,v^{\prime})>0,$ so we further obtain%
\[
d(x,v)<d(x^{\prime},v^{\prime}).
\]
Assuming that $x$ and $x^{\prime}$ are not distinguished by $s$ either, would
by symmetry yield $d(x,v)>d(x^{\prime},v^{\prime}).$ These two inequalities
give a contradiction. Therefore, $x$ and $x^{\prime}$ are distinguished either
by $s$ or $s^{\prime},$ so the claim is proven.

\medskip

From Claims A and B it follows that a pair $x$ and $x^{\prime}$ of $G$ is not
distinguished by $S$ only if $x$ belongs to a non-trivial block $G_{i}$ and
$x^{\prime}$ belongs to a non-trivial block $G_{j}$ such that $G_{i}$ and
$G_{j}$ share a cut vertex $v.$ We say that such a vertex $v$ is
\emph{critical} on both $G_{i}$ and $G_{j},$ and blocks $G_{i}$ and $G_{j}$
are said to be $v$\emph{-incident}. We say that a vertex (resp. an edge) $x$
in $G_{i}$ is $v$\emph{-critical} if a shortest path from $x$ to every vertex
from $S_{i}$ leads through a critical vertex $v.$ Notice that all $v$-critical
vertices in $G_{i}$ induce a path in $G_{i}$ starting at $v,$ otherwise
$S_{i}$ would not be a vertex (resp. an edge) metric generator in $G_{i}.$ We
call such a path a $v$\emph{-path} in $G_{i}.$ If $G_{i}$ and $G_{j}$ are
$v$-incident, then each of them contains a $v$-path $P_{i}$ and $P_{j}$
respectively, and there are pairs of vertices (resp. edges) belonging to
$V(P_{i})\cup V(P_{j})$ (resp. $E(P_{i})\cup E(P_{j})$) which are not
distinguished by $S$.

Notice the following: if we denote by $S^{\prime}$ a set obtained from $S$ by
introducing to it a vertex from $(V(P_{i})\cup V(P_{j}))\backslash\{v\}$ then
every pair of vertices from $V(P_{i})\cup V(P_{j})$ (resp. edges from
$E(P_{i})\cup E(P_{j})$) will be distinguished by $S^{\prime}$ and the
critical $v$-incidence of $G_{i}$ and $G_{j}$ will be broken. Therefore, in
order to obtain a vertex (resp. an edge) metric generator $S^{\ast}$ in $G,$
every critical incidence of blocks in $G$ must be broken this way. The only
question is what is the smallest number of vertices that need to be introduced
to $S$ in order to break all critical incidences of the blocks in $G.$ To
answer this question, let us consider the following construction.

For a graph $G$, let us define its corresponding graph $\Gamma$ in a following
manner. Let $\mathcal{G=\{}G_{1},\ldots,G_{q}\mathcal{\}}$ be a set of all
non-trivial blocks in $G$ and let $\mathcal{V}$ be the set of all critical
vertices $v$ in $G.$ We define $V(\Gamma)=\mathcal{G}\cup\mathcal{V}$ and
$E(\Gamma)$ consists of all edges $G_{i}v$ where $G_{i}\in\mathcal{G},$
$v\in\mathcal{V}$ and $v$ is critical on $G_{i}.$ The construction of $\Gamma$
from $G$ is illustrated by Figure \ref{Fig_dual}. Notice that $\Gamma$ is a
forrest in which all leaves are from $\mathcal{G}$. The open neighborhood of
$v\in\mathcal{V}$ in $\Gamma$ represents all blocks in $G$ which are pairwise
$v$-incident.

\bigskip\noindent\textbf{Claim C.} \emph{There exists a set }$E^{\prime
}\subseteq E(\Gamma)$\emph{ such that }$\left\vert E^{\prime}\right\vert \leq
q-1$\emph{ and every vertex }$v\in\mathcal{V}$\emph{ is incident to at most
one edge from }$E(\Gamma)\backslash E^{\prime}.$

\smallskip\noindent First, assume that $\Gamma$ is connected. We start with
$E^{\prime}=\phi.$ Let $G_{i}$ be a vertex of the maximum degree in $\Gamma$
among vertices from $\mathcal{G}$. We designate $G_{i}$ to be a root vertex of
$\Gamma,$ and let $N(G_{i})\subseteq\mathcal{V}$ denote the open neighborhood
of $G_{i}$ in $\Gamma.$ For every $v\in N(G_{i})$ we introduce to $E^{\prime}$
all edges from $\Gamma$ incident to $v$ except $vG_{i}.$ Notice that for each
neighbor $v$ of the root $G_{i}$ it holds that it is incident to at most one
edge not included in $E^{\prime}$. The procedure is then applied repeatedly on
all trees from $\Gamma-(\{G_{i}\}\cup N(G_{i})),$ where in every such tree we
designate as root the only neighbor of $v\in N(G_{i})$ contained in that tree.
The set $E^{\prime}$ obtained by this procedure is illustrated by Figure
\ref{Fig_dual}.

Since $\Gamma$ is bipartite with partition $(\mathcal{G},\mathcal{V}),$ every
edge from $E^{\prime}$ is incident to precisely one vertex from $\mathcal{G}$.
Also, by the construction every vertex from $G$ is incident to one edge in
$E^{\prime},$ except the initial root $G_{i}$. Therefore, $\left\vert
E^{\prime}\right\vert =q-1.$ Every vertex $v\in\mathcal{V}$ will be a neighbor
of a designated root in one step of the procedure, so it is incident to at
most one edge from $E(\Gamma)\backslash E^{\prime}$. If $\Gamma$ is not
connected, then $\Gamma$ is a forest and the same argument can be applied to
each of its connected components, by which we obtain $\left\vert E^{\prime
}\right\vert <q-1,$ which concludes the proof of Claim C.

\medskip

We may consider that an edge $G_{i}v$ from $\Gamma$ represents the $v$-path
$P_{i}$ on $G_{i}$ or, more specifically, a vertex from $P_{i}.$ So, let
$E^{\prime}$ be as in Claim C and let $S^{\prime}\subseteq V(G)$ be a set
consisting of a vertex from every $v$-path $P_{i}\in E^{\prime}.$ This implies
that introducing $S^{\prime}$ into $S,$ all pairwise critical incidences of
blocks around a critical vertex $v$ will be broken, and so for every critical
vertex $v$ of $G.$ This implies that $S^{\ast}=S\cup S^{\prime}$ is a vertex
metric generator in $G.$ Since%
\[
\left\vert S\right\vert \leq\left\vert S\right\vert +\left\vert S^{\prime
}\right\vert =\mathrm{dim}(G_{1})+\cdots+\mathrm{dim}(G_{q})-(q-p)+q-1,
\]
the proof is finished.
\end{proof}

\begin{figure}[h]
\begin{center}
\includegraphics[scale=0.75]{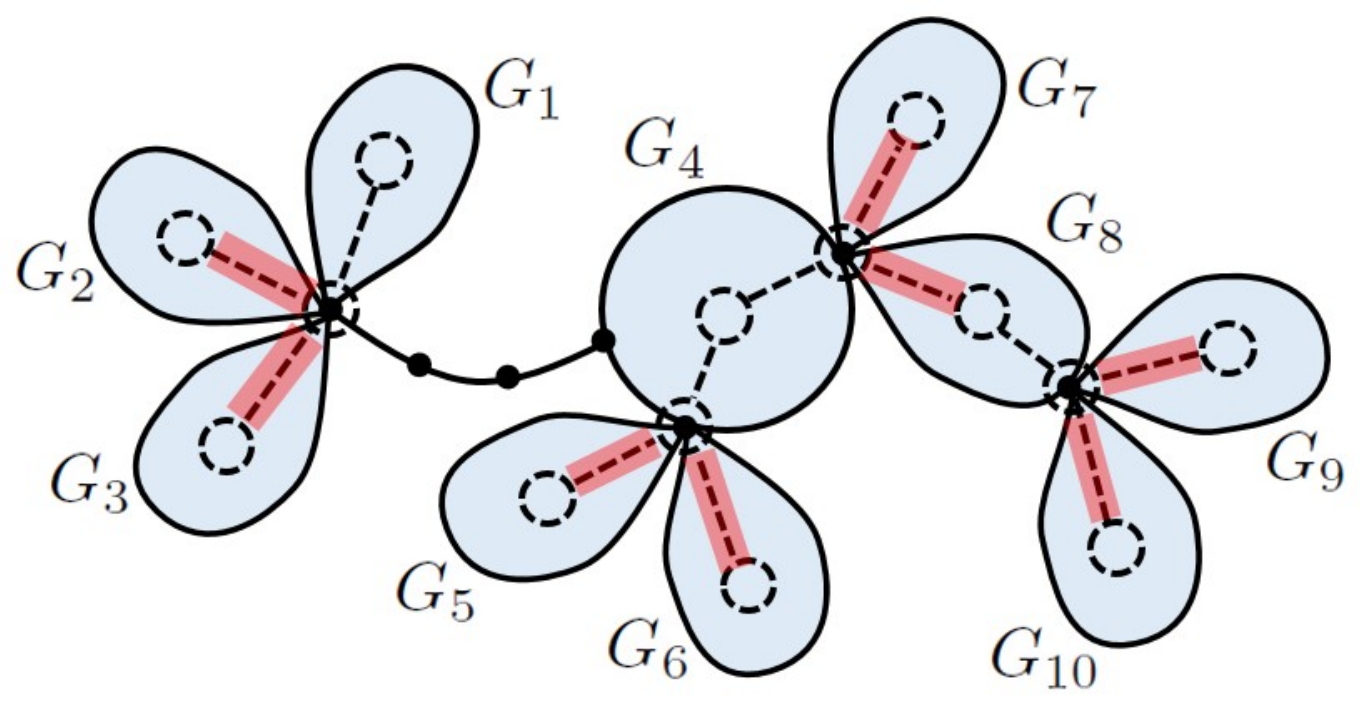}
\end{center}
\caption{A graph $G$ with $\delta(G)\geq2$ and ten non-trivial blocks $G_{i}.$
Assuming that all blocks which share a vertex $v$ are $v$-incident, the graph
$\Gamma$ of $G$ is shown in dashed line. Graph $\Gamma$ is a forest with two
trees, by designating $G_{1}$ and $G_{4}$ as the roots of the corresponding
trees of $G,$ a set $E^{\prime}\subseteq E(\Gamma)$ from Claim C within the
proof of Theorem \ref{Tm_blocks} is marked in $\Gamma.$}%
\label{Fig_dual}%
\end{figure}

Now we can use the previous two lemmas to state a main result of this section.

\begin{theorem}
\label{Tm_blocks}Let $G$ be a graph with $\delta(G)\geq2.$ Let $G_{1}%
,\ldots,G_{q}$ be all non-trivial blocks in $G.$ Suppose that $\mathrm{dim}%
(G_{i})\leq2c(G_{i})-1$ (resp. $\mathrm{edim}(G_{i})\leq2c(G_{i})-1$) whenever
$G_{i}$ is not a cycle. Then $\mathrm{dim}(G)\leq2c(G)-1$ (resp.
$\mathrm{edim}(G)\leq2c(G)-1$).
\end{theorem}

\begin{proof}
Assume that non-trivial blocks $G_{i}$ of $G$ are denoted so that $G_{i}$ is a
cycle if and only if $i>p.$ Then by Lemma \ref{Lemma_dim} we have%
\[
\mathrm{dim}(G)\leq\mathrm{dim}(G_{1})+\cdots+\mathrm{dim}(G_{q})+p-1.
\]
Since the metric dimension of the cycle equals two and since we assumed
$\mathrm{dim}(G_{i})\leq2c(G_{i})-1$ for every $i=1,\ldots,p,$ we further
obtain%
\begin{align*}
\mathrm{dim}(G)  &  \leq(2c(G_{1})-1)+\cdots+(2c(G_{p})-1)+2(q-p)+p-1=\\
&  =2c(G_{1})+\cdots+2c(G_{q})-p+(q-p)+p-1.
\end{align*}
Lemma \ref{Lemma_c} now yields
\[
\mathrm{dim}(G)\leq2c(G)-(q-p)-1\leq2c(G)-1,
\]
and we are finished. The proof for $\mathrm{edim}(G)$ is analogous.
\end{proof}

As for the question when the equality is attained, the proofs of Lemma
\ref{Lemma_dim} and Theorem \ref{Tm_blocks} imply the following necessary condition.

\begin{corollary}
Let $G\not =C_{n}$ be a graph with $\delta(G)\geq2.$ If $\mathrm{dim}%
(G_{i})<2c(G_{i})-1$ (resp $\mathrm{edim}(G)\leq2c(G)-1$) for a block $G_{i}$
of $G$ distinct from a cycle or there exist two vertex-disjoint non-trivial
blocks $G_{i}$ and $G_{j}$ in $G$, then $\mathrm{dim}(G)<2c(G)-1$ (resp.
$\mathrm{edim}(G)<2c(G)-1$).
\end{corollary}

As for the role of Theorem \ref{Tm_blocks} in the journey towards the solution
of Conjectures \ref{Con_dim_leaves} and \ref{Con_edim_leaves} we can state the following.

\begin{corollary}
If Conjecture \ref{Con_dim_leaves} (resp. Conjecture \ref{Con_edim_leaves})
holds for all $2$-connected graphs $G$ distinct from $C_{n},$ then it holds in general.
\end{corollary}

\begin{proof}
The claim is the consequence of Theorem \ref{Tm_blocks} and the fact that
every non-trivial block in $G$ with $\delta(G)\geq2$ is a $2$-connected graph.
\end{proof}

\section{Concluding remarks}

In \cite{SedSkreExtensionCactus} it was established that the inequalities
$\mathrm{dim}(G)\leq L(G)+2c(G)$ and $\mathrm{edim}(G)\leq L(G)+2c(G)$ hold
for cacti and it was further conjectured that the same upper bounds hold for
metric dimensions of general connected graphs. Noticing that the attainment of
these bounds in the class of cacti depends on the presence of the leaves in a
graph, in this paper we focused on leafless graphs. In leafless graphs it
holds that $L(G)=0,$ so the conjectured upper bound becomes $2c(G).$ We
started by characterizing all cacti for which the first smaller upper bound is
attained, i.e. $\mathrm{dim}(G)=L(G)+2c(G)-1$ and $\mathrm{edim}(G)\leq
L(G)+2c(G)-1.$ A direct consequence of this characterization is that there are
some leafless cacti for which this decreased bound is attained. Therefore, in
the class of leafless cacti the upper bound $2c(G)-1$ is tight for both metric dimensions.

We conjectured that the decreased upper bound $2c(G)-1$ hols for both metric
dimensions of all connected graphs without leaves, i.e. all graphs with
$\delta(G)\geq2$. These conjectures we tested both systematically and
stochastically for graphs of smaller order and we did not encounter a
counterexample. As a first step towards the solution of this conjecture, we
show that it holds for all graphs with $\delta(G)\geq3$. Also, we showed that
if the decreased bound $2c(G)-1$ hold for metric dimensions of $2$-connected
graph, then they will also hold for all graphs with $\delta(G)\geq2,$ i.e. the
conjecture will be solved if it is established to hold for $2$-connected
graph. Establishing that the conjectured bounds hold for metric dimensions of
$2$-connected graphs we leave as an open problem.

\bigskip

\bigskip\noindent\textbf{Acknowledgments.}~~Both authors acknowledge partial
support of the Slovenian research agency ARRS program\ P1-0383 and ARRS
project J1-1692. The first author also the support of Project
KK.01.1.1.02.0027, a project co-financed by the Croatian Government and the
European Union through the European Regional Development Fund - the
Competitiveness and Cohesion Operational Programme.

\end{document}